\documentclass{amsart}

\usepackage{amssymb}
\usepackage{amsfonts}
\usepackage{mathrsfs}
\usepackage{amsmath}

\newtheorem{theorem}{Theorem}[section]
\newtheorem{corollary}[theorem]{Corollary}
\newtheorem{lemma}[theorem]{Lemma}
\newtheorem{proposition}[theorem]{Proposition}

\numberwithin{equation}{section}

\newtheorem*{conjecture}{Conjecture}

\newcommand{\RR}{\mathbb{R}}
\newcommand{\CC}{\mathbb{C}}

\newcommand{\hyperg}[4]{\: _2\! F_1\! \left[\!\! \begin{array}{c} #1,\, #2 \\ #3 \end{array}\!\! ;\,  #4 \right]}

\newcommand{\RePt}{\mathrm{Re}\,}

\newcommand{\disk}{\mathbb{D}}

\begin{document}

%
%

\title[Norm of the Bergman projection]{Two-sided norm estimates for Bergman-type projections with an asymptotically sharp lower bound}
\thanks{\noindent The first author was supported by the National Natural Science
Foundation of China grant 11171318 and OATF,USTC; the second author was
supported by the Academy of Finland project number 79999201 and the Emil Aaltonen Foundation; the third author
was supported by the National Natural Science Foundation of China grant 11271124, 11301136,
and the Natural Science Foundation of Zhejiang province grant LQ13A010005.}


\author{Congwen Liu}


\address{School of Mathematical Sciences, University of Science and Technology of China, Hefei, Anhui 230026, People's Republic of China.}
\email{cwliu@ustc.edu.cn}

\author{Antti Per\"al\"a}

\address{Department of Physics and Mathematics, University of Eastern Finland, P.O. Box 111, 80101 Joensuu, Finland.}

\email{antti.perala@uef.fi}

\author{Lifang Zhou}

\address{Department of Mathematics,
Huzhou University,
Huzhou, Zhejiang 313000,
People¡¯s Republic of China}
\email{lfzhou@zjhu.edu.cn}

\begin{abstract}
We obtain new two-sided norm estimates for the family of Bergman-type projections arising from the standard weights $(1-|z|^2)^{\alpha}$ where $\alpha>-1$. As $\alpha\to -1$, the lower bound is sharp in the sense that it asymptotically agrees with the norm of the Riesz projection. The upper bound is estimated in terms of the maximal Bergman projection, whose exact operator norm we calculate. The results provide evidence towards a conjecture that was posed very recently by the first author.
\end{abstract}




\maketitle

%
%

\section{Introduction}

\noindent Let $\disk$ denote the unit disk in the complex plane and write $dA$ for the normalized Lebesgue area measure $dA(z)=\pi^{-1}dxdy$ (where $z=x+iy$). For $\alpha>-1$, the standard weighted area measure $dA_\alpha$ is given by
$$dA_\alpha(z)=(1+\alpha)(1-|z|^2)^{\alpha} dA(z).$$ Because of the normalization, $dA_\alpha$ is also a probability measure. As usual, for $p>0$, the space $L_{\alpha}^p(\disk)$ consists of all Lebesgue measurable
functions $f$ on $\disk$ for which
\begin{equation*}
\|f\|_{p,\alpha}:=\left\{\int_{\disk} |f(z)|^p dA_{\alpha}(z)\right\}^{\frac 1{p}}
\end{equation*}
is finite.  For $0<p<\infty$, the Bergman space $A_{\alpha}^p$ consists of analytic
functions $f$ in $L^p_\alpha(\disk)$, while $H^\infty(\disk)$ denotes the space of bounded analytic functions. Under the weighted $L^p$-topologies, the spaces above are complete, and Banach whenever they are normed.

We consider the natural projection onto these spaces, i.e., the orthogonal projection from $L_{\alpha}^{2}(\disk)$ onto
$A_{\alpha}^{2}$, also known as the Bergman projection. It can be expressed as an integral operator:
\[
P_{\alpha}f(z)=\int_{\disk} \frac {f(w)}{(1-z\bar{w})^{2+\alpha}} dA_{\alpha}(w).
\]
As usual, $dA_0=dA$, $L^p_0(\disk)=L^p(\disk)$, $A^p_0(\disk)=A^p(\disk)$, $\|\cdot\|_{p,0}=\|\cdot\|_p$ and $P_0=P$.

The Bergman projection is a central object in the study of analytic function spaces. It naturally relates to fundamental questions such as duality and harmonic conjugates, and it is also a building block for Toeplitz operators. Understanding its behaviour and estimating its size is therefore of vital importance on several occasions. There are several textbooks on Bergman spaces and Bergman projections. For the interested reader, we recommend \cite{DS, HKZ,  Zhu1, Zhu3}.

Let
\[
\|P_{\alpha}\|_{p,\alpha}:=\sup \left\{ \frac {\|P_{\alpha} f\|_{p,\alpha}} {\|f\|_{p,\alpha}}:\; f\in L_{\alpha}^p, f\neq 0\right\}
\]
be the operator norm of $P_{\alpha}$. For $p=2$, $P_\alpha$ is an orthogonal projection on $L^2_\alpha$, and hence $\|P_{\alpha}\|_{2,\alpha}=1$.
The purpose of the present work is to obtain new and improved two-sided estimates for $\|P_\alpha\|_{p,\alpha}$ for the full range
of $p \in (1,\infty)$, $\alpha>-1$. Our main result reads as follows.
\begin{theorem}\label{thm:main1}
For $1< p<\infty$, we have
\begin{equation}\label{eqn:main}
\frac{\Gamma(\frac {2+\alpha}{p})\Gamma(\frac {2+\alpha}{q})}{\Gamma^2 (\frac {2+\alpha}{2})}
\leq~ \|P_{\alpha}\|_{p,\alpha} ~\leq~ (1+\alpha) \frac{\Gamma(\frac {1+\alpha}{p})\Gamma(\frac {1+\alpha}{q})}{\Gamma^2 (\frac {2+\alpha}{2})},
\end{equation}
where $q:=\frac {p}{p-1}$ is the conjugate exponent of $p$.
\end{theorem}

The boundedness of $P_\alpha$ when $\alpha=0$ dates back fifty years and is due to Zaharjuta and Judovi\v c \cite{ZJ}. About ten years later, the result was established for the whole scale $\alpha>-1$ by Forelli and Rudin \cite{FR}. However, calculating the exact value of $\|P_{\alpha}\|_{p,\alpha}$ has turned out to be very challenging. Zhu \cite{Zhu2} first obtained the right asymptotics of $\|P_{\alpha}\|_{p,\alpha}$, and later Dostani\'c \cite{Dos2} found a quantitative version of Zhu's result for $\alpha=0$. Dostani\'c also conjectured that
\begin{equation}\label{eq:dos}
\|P\|_p=\csc(\pi/p).
\end{equation}

Note that the conjecture agrees with the already established value for the norm of the Riesz projection due to Hollenbeck and Verbitsky \cite{HV}. However, very recently, in \cite{Liu},  the first author of the present paper obtained
$$\|P\|_p\geq \Gamma(2/p)\Gamma(2/q),$$
which disproves \eqref{eq:dos}. However, this motivated the following conjecture.

\begin{conjecture}[Liu]\label{liu:conj} For $1<p<\infty$, we have
$$\|P\|_p= \Gamma(2/p)\Gamma(2/q).$$
\end{conjecture}

One of the reasons for writing this paper is to provide evidence supporting the above conjecture. Namely, observe that the lower bound in \eqref{eqn:main} tends to $\csc(\pi/p)$ as $\alpha\to -1$. Since $\alpha=-1$ can often be viewed as the case of the Hardy spaces and the Riesz projection, we have convincing evidence that
$$\frac{\Gamma(\frac {2+\alpha}{p})\Gamma(\frac {2+\alpha}{q})}{\Gamma^2 (\frac {2+\alpha}{2})}
= \|P_{\alpha}\|_{p,\alpha},$$ which we conjecture to be true. When $\alpha=0$, we then recover the conjecture from \cite{Liu}. The lower bound is obtained by using a suitable choice of test functions formed from Bergman-type kernels along with some interpolation, manipulation of the classical Forelli-Rudin estimates \cite{FR}, and a Hausdorff-Young type inequality. We remark that in our argument there is a cut-off at $\alpha=(2-p)/(1-p)$, and we need two separate methods. It would be of some interest to find a unified approach that directly covers all cases.

For the upper bound, we consider the maximal Bergman projection, which is arguably a quite standard approach in this direction. However, we manage to calculate its exact operator norm. For the unweighted case, this result can be deduced from the work of Dostani\'c \cite{Dos} -- the weighted case is probably new.

In the recent years, there has been increasing interest in the study of the size of Bergman projections in various context other than $A^p_\alpha$. For the Bloch space, we mention the works of the second author \cite{Per1, Per2}, as well as the work of Kalaj-Markovi\'c \cite{KM}. For the Besov spaces we mention the papers of Kaptano\u glu-\"Ureyen, Per\"al\"a and Vujadinovi\'c \cite{KU,Per3,Vuj}.

\section{Preliminaries}

\noindent We use the classical notation for the functions $_2\! F_1$ 
\begin{equation*}\label{eq:hypergdefin}
\hyperg{a}{b}{c}{\lambda}=\sum_{k=0}^{\infty}\frac{(a)_k(b)_k}{(c)_k}\frac{\lambda^k}{k!}
\end{equation*}
with $c\neq 0, -1,-2,\ldots$, where
\[
(a)_0=1,\quad(a)_k=a(a+1)\ldots(a+k-1)
\quad \text{ for } k\geq1.
\]
denotes the Pochhammer symbol of $a$. This series gives an analytic function for $|\lambda|<1$, called the Gauss hypergeometric function associated to
$(a,b,c)$.

We refer to \cite[Chapter II]{Erd} for the properties of these
functions. Here, we only list some of them for later reference.
\begin{align}
&\!\!\!\hyperg{a}{b}{c}{1^-} = \frac {\Gamma(c) \Gamma(c-a-b)}
{\Gamma(c-a) \Gamma(c-b)},\qquad \RePt(c-a-b)>0.
\label{eqn:gauss}\\
&\hyperg{a}{b}{c}{\lambda} = (1-\lambda)^{c-a-b} \hyperg{c-a}{c-b}{c}{\lambda}. \label{eqn:euler}\\
&\hyperg{a}{b}{c}{\lambda}=\frac{\Gamma(c)}{\Gamma(b)\Gamma(c-b)}\int_0^1
t^{b-1}(1-t)^{c-b-1} (1-t\lambda)^{-a} dt,\label{eqn:euler2}\\
&\hspace{64pt} \RePt c>\RePt b>0;\;  |\arg(1-\lambda)|<\pi;\;
\lambda\neq 1.\notag\\
&\hyperg{a}{b}{c}{\lambda}=\frac{\Gamma(c)}{\Gamma(\gamma)\Gamma(c-\gamma)}\int_0^1
t^{\gamma-1}(1-t)^{c-\gamma-1}\hyperg{a}{b}{\gamma}{t\lambda}dt,\label{eq:integralhyperg}\\
&\hspace{64pt} \RePt c>\RePt \gamma>0;\;  |\arg(1-\lambda)|<\pi;\;
\lambda\neq 1.\notag \\
\frac {d^k}{d\lambda^k} &\hyperg{a}{b}{c}{\lambda} = \frac {(a)_k (b)_k}{(c)_k} \hyperg {a+k}{b+k}{c+k}{\lambda}, \quad k\in \mathbb{N}.
\label{eqn:diffhyperg}
\end{align}

\begin{lemma}\label{lem:aneasyformula}
Suppose $\RePt c>0$, $\RePt \delta>0$ and $\RePt (\delta+c-a-b)>0$. Then
\begin{equation}\label{eqn:aneasyformula}
\int_0^1 t^{c-1}(1-t)^{\delta-1}\hyperg{a}{b}{c}{t}dt = \frac {\Gamma(c)
\Gamma(\delta)\Gamma(\delta+c-a-b)} {\Gamma(\delta+c-a) \Gamma(\delta+c-b)}.
\end{equation}
\end{lemma}

\begin{proof}
Note that, under the assumption of the lemma, both sides of \eqref{eq:integralhyperg} are continuous at $\lambda=1$.
The lemma then follows by letting $\lambda\to 1^{-}$ and applying \eqref{eqn:gauss}.
\end{proof}

\begin{lemma}\label{lem:crucial}
Let $a,b,c\in \mathbb{R}$ and $t>-1$. The identity
\begin{align}\label{eqn:crucial2}
&\int_{\disk} \frac {(1-|\xi|^2)^t dA(\xi)} {(1-z\bar{\xi})^{a} (1-w\bar{\xi})^{b} (1-\xi\bar{w})^{c}}\\
&\qquad =~ \frac {1} {1+t} \sum_{j=0}^{\infty} \frac {(a)_{j} (c)_{j}} {(2+t)_{j} j!} \hyperg {b}{c+j}{2+t+j}{|w|^2}
(z\bar{w})^j \notag
\end{align}
holds for any $z,w\in \disk$.
\end{lemma}

\begin{proof}
Recall that
\begin{equation}\label{eqn:binom}
(1-\lambda)^{-\gamma} = \sum_{k=0}^{\infty} \frac {(\gamma)_k}{k!} \lambda^k
\end{equation}
holds for $\lambda\in \mathbb{C}$, $|\lambda|<1$ and $\gamma\in \mathbb{R}$.
This leads to
\begin{align*}
&\frac {1}{2\pi} \int_0^{2\pi} (1-z e^{-i\theta})^{-a} (1-w e^{-i\theta})^{-b} (1-e^{i\theta} \bar{w})^{-c} d\theta\\
& \qquad =~ \sum_{j=0}^{\infty} \sum_{k=0}^{\infty} \sum_{\ell=0}^{\infty} \frac {(a)_{j} (b)_{k} (c)_{\ell}} {j!k!\ell !}
\frac {1}{2\pi} \int_0^{2\pi}(z e^{-i\theta})^j (w e^{-i\theta})^k (\bar{w}e^{i\theta})^{\ell} d\theta\\
& \qquad =~ \sum_{j=0}^{\infty} \sum_{k=0}^{\infty} \frac {(a)_{j} (b)_{k} (c)_{j+k}} {j!k! (j+k)!}
\, |w|^{2k} (z\bar{w})^j.
\end{align*}
Note that $(c)_{j+k}=(c+j)_k (c)_j$ and $(j+k)!=(1+j)_k (1)_j$. Then
\begin{align*}
&\frac {1}{2\pi} \int_0^{2\pi} (1-z e^{-i\theta})^{-a} (1-w e^{-i\theta})^{-b} (1-e^{i\theta} \bar{w})^{-c} d\theta\\
& \qquad =~ \sum_{j=0}^{\infty} \frac {(a)_{j} (c)_{j}} {(1)_j j!}
\left\{ \sum_{k=0}^{\infty} \frac {(b)_{k} (c+j)_{k}} {(1+j)_k k!}  |w|^{2k} \right\} (z \bar{w})^j\\
& \qquad =~ \sum_{j=0}^{\infty} \frac {(a)_{j} (c)_{j}} {(1)_j j!} \hyperg {b}{c+j}{1+j}{|w|^2} (z \bar{w})^j.
\end{align*}
Therefore
\begin{align*}
&\int_{\disk} \frac {(1-|\xi|^2)^t dA(\xi)} {(1-z\bar{\xi})^{a} (1-w\bar{\xi})^{b} (1-\xi\bar{w})^{c}}\\
&\qquad =~ 2 \int_0^1 \left\{ \frac {1}{2\pi} \int_0^{2\pi} (1-z re^{-i\theta})^{-a} (1-w re^{-i\theta})^{-b}
(1-re^{i\theta} \bar{w})^{-c} d\theta  \right\} (1-r^2)^t r dr \\
&\qquad =~ 2 \sum_{j=0}^{\infty} \frac {(a)_{j} (c)_{j}} {(1)_{j} j!} \left\{ \int_0^1 r^{1+2j} (1-r^2)^t
\hyperg {b}{c+j}{1+j}{r^2 |w|^2} dr \right\} (z\bar{w})^j.
\end{align*}
By \eqref{eq:integralhyperg}, the integral in the parentheses equals
\[
\frac 12 \frac {\Gamma(1+j) \Gamma(1+t)} {\Gamma(2+t+j)} \hyperg {b} {c+j} {2+t+j} {|w|^2}.
\]
Inserting this into the above series yields \eqref{eqn:crucial2}.
\end{proof}

As an immediate consequence, we have the following.

\begin{corollary}\label{cor:rudin1}
For $a\in \RR$ and $t>-1$, we have
\begin{align}\label{eqn:keyfml2a}
\int_{\disk} \frac{(1-|\xi|^2)^{t}}{|1-z\bar{\xi}|^{2a}} dA(\xi)
=\frac{1}{1+t} \hyperg{a}{a}{2+t}{|z|^2}
\end{align}
holds for all $z\in \disk$.
\end{corollary}

\begin{corollary}[Forelli-Rudin estimates, see, for instance, {\cite[p.7, Theorem 1.7]{HKZ}}]\label{cor:rudin2}
For $a\in \RR$ and $t>-1$, we have
\begin{align}\label{eqn:forelli-rudin}
\int_{\disk} \frac{(1-|\xi|^2)^{t}}{|1-z\bar{\xi}|^{2+t+c}} dA(\xi)
~\approx~  \begin{cases}
1, & \text{if } c<0;\\
\log \dfrac {1}{1-|z|^2}, & \text{if } c=0;\\
(1-|z|^2)^{-c}, & \text{if } c>0.
\end{cases}
\end{align}
as $|z|\to 1^{-}$. Here, we use the symbol $\approx$ to indicate that two quantities have the same behavior
asymptotically.
\end{corollary}

\begin{corollary}
Let $t>-1$ and $a>1+t/2$. We have
\begin{equation}\label{eqn:supval}
\sup_{z\in \disk} \left\{ (1-|z|^2)^{2a-t-2} \int_{\disk} \frac{(1-|w|^2)^{t}dA(w)} {|1-z\bar{w}|^{2a}}\right\}
~=~ \frac {\Gamma(1+t)\Gamma(2a-t-2)}{\Gamma^2 (a)}.
\end{equation}
\end{corollary}

\begin{proof}
By \eqref{eqn:keyfml2a} and \eqref{eqn:euler}, we have
\begin{align*}
(1-|z|^2)^{2a-t-2} \int_{\disk} \frac{(1-|w|^2)^{t}dA(w)} {|1-z\bar{w}|^{2a}}
= \frac {1}{1+t} \hyperg {2+t-a} {2+t-a} {2+t}{|z|^2}.
\end{align*}
Note that the last hypergeometric function is increasing in the
interval $[0, 1)$, since its Taylor coefficients are all positive.
It follows that
\begin{align*}
\sup_{z\in \disk} \hyperg {2+t-a} {2+t-a} {2+t}{|z|^2} =&
\hyperg {2+t-a} {2+t-a} {2+t}{1^-}\\
=&\frac {\Gamma(2+t) \Gamma(2a-2-t)} {\Gamma^2 (a)}.
\end{align*}
This gives \eqref{eqn:supval}.
\end{proof}

\begin{corollary}\label{cor:rudin2}
Suppose that $a, b>0$, $c\in \mathbb{R}$, and $1+a+b-2c>0$. Then
\begin{align}\label{eqn:intglval}
\int_{\disk} |z|^{2b}& (1-|z|^2)^{a-1} \left\{\int_{\disk} \frac {(1-|w|^2)^{b-1}}
{|1-z\bar{w}|^{2c}}  dA(w)\right\}dA(z)\\
&=~\frac {\Gamma(a)\Gamma(b)\Gamma(1+a+b-2c)}
{\Gamma^2 \left(1+a+b-c\right)}. \notag
\end{align}
\end{corollary}

\begin{proof}
Using \eqref{eqn:keyfml2a} in the inner integral, the left-hand side of  \eqref{eqn:intglval} equals
\begin{align*}
\frac {1}{b}& \int_{\disk} |z|^{2b} (1-|z|^2)^{a-1}
\hyperg {c}{c}{1+b}{|z|^2} dA(z)\\
& \quad = \frac {1}{b} \int_0^1 r^{b} (1-r)^{a-1}
\hyperg {c}{c}{1+b}{r} dr.
\end{align*}
Now \eqref{eqn:intglval} follows from an application of Lemma \ref{lem:aneasyformula}.
\end{proof}

\section{The proof of Theorem \ref{thm:main1}: the upper estimate}

\noindent We consider the ``maximal Bergman projection''
\[
P_{\alpha}^{\sharp}f(z)=\int_{\disk} \frac {f(w)}{|1-z\bar{w}|^{2+\alpha}} dA_{\alpha}(w).
\]
It is clear that $\|P_{\alpha}\|_{p,\alpha}\leq \|P_{\alpha}^{\sharp}\|_{p,\alpha}$, so it suffices to show the following.
\begin{proposition}\label{thm:maximal}
For $1<p<\infty$ and $\alpha>-1$, we have
\[
\|P_{\alpha}^{\sharp} \|_{p,\alpha} = \frac {(1+\alpha)\Gamma(\frac {1+\alpha}{p})\Gamma(\frac {1+\alpha}{q})}{\Gamma^2 \left(\frac {2+\alpha}{2}\right)}.
\]
\end{proposition}

We appeal to the well known Schur's test (see, for instance, \cite[Theorem 3.6]{Zhu3}).
\begin{lemma}\label{lem:shurtest}
Suppose that $(X,\mu)$ is a $\sigma$-finite measure space and
$K(x,y)$ is a nonnegative measurable function on $X\times X$ and $T$
the associated integral operator
\[
Tf(x)=\int_X K(x,y) f(y) d\mu(y).
\]
Let $1<p<\infty$ and $1/p+1/q=1$. If there exist a positive constant
$C$ and a positive measurable function $u$ on $X$ such that
\[
\int_X K(x,y) u(y)^{q} d\mu(y) \leq  C u(x)^{q}
\]
for almost every $x$ in $X$ and
\[
\int_X K(x,y) u(x)^{p} d\mu(x) \leq  C u(y)^{p}
\]
for almost every $y$ in $X$, then $T$ is bounded on $L^p(X,d\mu)$ with
$\|T\|\leq C$.
\end{lemma}

\begin{proof}[Proof of Proposition \ref{thm:maximal}]
With
\begin{align*}
K(z,w)=\frac {1}{|1-z\bar{w}|^{2+\alpha}},\\
u(z)=(1-|z|^2)^{-\frac {1+\alpha}{pq}},
\end{align*}
where $q$ is the conjugate exponent of $p$, it is clear that
\begin{equation*}
\int_{\disk}K(z,w) u(w)^{q} dA_{\alpha}(w)\leq C(p) u(z)^q
\end{equation*}
for almost every $z\in \disk$ and
\begin{equation*}
\int_{\disk}K(z,w) u(z)^{p} dA_{\alpha}(z)\leq C(q) u(w)^p
\end{equation*}
for almost every $w\in \disk$. Here
\[
C(p)~=~(1+\alpha) \sup_{z\in \disk} \left\{(1-|z|^2)^{\frac {1+\alpha}{p}} \int_{\disk}\frac{(1-|w|^2)^{-\frac {1+\alpha}{p}+\alpha} dA(w)}
{|1-z\bar{w}|^{2+\alpha}}\right\}.
\]
In view of \eqref{eqn:supval}, we find that
\[
C(p)=(1+\alpha)\frac {\Gamma(\frac {1+\alpha}{p})\Gamma(\frac {1+\alpha}{q})}{\Gamma^2 \left(\frac {2+\alpha}{2}\right)}=C(q).
\]
Thus, an application of Schur's test gives
\[
\|P_{\alpha}^{\sharp} \|_{p,\alpha} \leq\frac {(1+\alpha)\Gamma(\frac {1+\alpha}{p})\Gamma(\frac {1+\alpha}{q})}{\Gamma^2 \left(\frac {2+\alpha}{2}\right)}.
\]

To prove the converse inequality, we define, for $\epsilon>0$,
\begin{align*}
g_{\epsilon}(w)~:=~&\epsilon^{1/p} (1-|w|^2)^{(\epsilon-1)(1+\alpha)/p}, \\
h_{\epsilon}(z)~:=~&\left\{ \frac {\Gamma(2+\alpha+\epsilon(1+\alpha) q)} {(1+\alpha)\Gamma(\epsilon (1+\alpha))
\Gamma(2+\alpha+ \epsilon (1+\alpha)q/p)} \right\}^{1/q}\\
&\qquad \qquad \times |z|^{2+2\alpha+2(\epsilon-1)(1+\alpha)/p} (1-|z|^2)^{(\epsilon-1)(1+\alpha)/q}.
\end{align*}
Easy calculations show that
$\|g_{\epsilon}\|_{p,\alpha}=\|h_{\epsilon}\|_{q,\alpha}=1$.

Applying Corollary \ref{cor:rudin2}, with $a=(1+\alpha)[1+(\epsilon-1)/q]$, $b=(1+\alpha)[1+(\epsilon-1)/p]$ and $c=1+\alpha/2$, we obtain
\begin{align*}
\int_{\disk}&\left\{\int_{\disk}  \frac{g_{\epsilon} (w)} {|1-z\bar{w}|^{2+\alpha}}dA_{\alpha}(w)\right\}\overline{h_\epsilon(z)}dA_{\alpha}(z)\\
&\quad \quad =~ (1+\alpha)^2 \times \frac{\Gamma\left(\frac {1+\alpha}{p}+ \frac {\epsilon(1+\alpha)}{q}\right)
\Gamma\left(\frac {1+\alpha}{q}+ \frac {\epsilon(1+\alpha)}{p}\right) \Gamma(\epsilon(1+\alpha))}
{\Gamma^2\left(\frac {2+\alpha}{2}+\epsilon(1+\alpha)\right)}\\
&\qquad \qquad\times \epsilon^{1/p} \left\{ \frac {\Gamma(2+\alpha+\epsilon(1+\alpha) q)} {(1+\alpha)\Gamma(\epsilon (1+\alpha))
\Gamma(2+\alpha+ \epsilon (1+\alpha)q/p)} \right\}^{1/q} .
\end{align*}
Having in mind that
\begin{align*}\label{eq:norm}
\|P^{\sharp}\|_p =\sup_{\substack{\|f\|_{p,\alpha}=1\\ \|g\|_{q,\alpha}=1}} \left|\int_{\disk}
\left(\int_{\disk} \frac{f(w)}{|1-z\bar{w}|^{2+\alpha}} dA_{\alpha}(w)\right)\overline{g(z)}dA_{\alpha}(z)\right|,
\end{align*}
this implies
\begin{align*}
\|P^{\sharp}\|_{p,\alpha}&~\geq~(1+\alpha)^2 \times \frac{\Gamma\left(\frac {1+\alpha}{p}+ \frac {\epsilon(1+\alpha)}{q}\right)
\Gamma\left(\frac {1+\alpha}{q}+ \frac {\epsilon(1+\alpha)}{p}\right) \Gamma(\epsilon(1+\alpha))}
{\Gamma^2\left(\frac {2+\alpha}{2}+\epsilon(1+\alpha)\right)}\\
&\qquad \qquad\times \epsilon^{1/p} \left\{ \frac {\Gamma(2+\alpha+\epsilon(1+\alpha) q)} {(1+\alpha)\Gamma(\epsilon (1+\alpha))
\Gamma(2+\alpha+ \epsilon (1+\alpha)q/p)} \right\}^{1/q} .
\end{align*}
The proof is completed by letting $\epsilon\rightarrow 0^{+}$.
\end{proof}

\section{The proof of Theorem \ref{thm:main1}: the lower estimate}

\noindent We proceed to show
\begin{equation}\label{eqn:lowerbd}
\|P_\alpha\|_{p,\alpha} \geq \frac {\Gamma\left(\frac{2+\alpha}{p}\right) \Gamma\left(\frac{2+\alpha}
{q} \right)}{\Gamma^2\left(\frac{2+\alpha}{2}\right)}.
\end{equation}
We only need to consider the case when $p>2$, and the case when $1<p<2$ then follows from the
duality. Curiously, it turns out that we cannot deal with the whole range $\alpha \in (-1,\infty)$ using the same argument. Indeed, we need separate arguments for the cases $$\alpha<(2-p)/(1-p)$$ and $$\alpha>(2-p)/(1-p).$$ Note that we can assume that $\alpha\neq (2-p)/(p-1)$, since when $\alpha = (2-p)/(p-1)$,
\eqref {eqn:lowerbd} can be easily derived from the other case. To see this, we begin with the following
\begin{lemma}\label{lem:normmonot}
The function $p\mapsto \|P_\alpha\|_{p,\alpha}$ is increasing on $[2,\infty)$.
\end{lemma}

\begin{proof}
Assume that $2<p_1<p_2$. Then, by the Riesz-Thorin interpolation theorem (see \cite{Gra},
p.34, Theorem 1.3.4), we have
\[
\|P_\alpha\|_{p_1,\alpha}\leq \|P_\alpha\|_{2,\alpha}^{1-\theta}\|P_\alpha\|_{p_2,\alpha}^\theta
\]
where $\theta$ is given by the relation
\[
\frac {1}{p_1}=\frac {1-\theta}{2}+\frac {\theta}{p_2}.
\]
Having in mind that $\|P_\alpha\|_{2,\alpha}=1$, $\|P_\alpha\|_{p_2,\alpha}\geq 1$ and $\theta \in (0,1)$, we get
\[
\|P_\alpha\|_{p_1,\alpha}\leq \|P_\alpha\|_{p_2,\alpha}^\theta\leq \|P_\alpha\|_{p_2,\alpha}.
\]
\end{proof}

Assume now that we have shown \eqref{eqn:lowerbd} for all $p>2$ and $\alpha\neq (2-p)/(p-1)$, or in other words, \eqref{eqn:lowerbd} holds for all
$2<p\neq p^{\ast}:=(2+\alpha)/(1+\alpha)$.
Thus, by Lemma \ref{lem:normmonot}, we have
\[
\|P_{\alpha}\|_{p^{\ast},\alpha} \geq \|P_{\alpha}\|_{p^{\ast}-\epsilon,\alpha} \geq \frac {\Gamma\left(\frac{2+\alpha}{p^{\ast}-\epsilon}\right)
\Gamma\left(2+\alpha-\frac{2+\alpha} {p^{\ast}-\epsilon} \right)}{\Gamma^2\left(\frac{2+\alpha}{2}\right)}
\]
for any $0<\epsilon<p^{\ast}-2$, which implies \eqref{eqn:lowerbd} is valid for $p=p^{\ast}$.

From now on, we assume $p>2$, $\alpha>-1$, $\alpha\neq (2-p)/(p-1)$ and let $\beta:=(2+\alpha)/2$.
We fix $\xi\in \disk$ and define
\[
f_{\xi}(z) := (1-\xi\bar{z})^{\beta-2\beta/p}  (1-z\bar{\xi})^{-\beta },
\qquad z\in \disk.
\]
Using \eqref{eqn:crucial2} we get
\begin{align*}
P_{\alpha}f_\xi(z)~=~& (2\beta-1)\int_{\disk}\frac{ (1-|w|^2)^{2\beta-2} dA(w)} {(1-z\bar{w})^{2\beta}
(1-\xi\bar{w})^{2\beta/p-\beta} (1-w\bar{\xi})^{\beta}}\\
=~& \sum_{k=0}^{\infty} \frac {\left(\beta \right)_k} {k!} \hyperg {2\beta/p - \beta}
{\beta +k } {2\beta+k} {|\xi|^2} (z\bar{\xi})^k.
\end{align*}
We now decompose
\begin{equation*}
P_{\alpha}f_{\xi}(z)=\Phi_{\xi}(z) + \Psi_{\xi}(z) + \Upsilon_{\xi}(z),
\end{equation*}
where
\begin{align}
\Phi_{\xi}(z) ~:=~& \frac{\Gamma(2\beta/p)\Gamma(2\beta/q)}{\Gamma^2 (\beta)}\,
(1-z\bar{\xi})^{-2\beta/p}, \label{eqn:fung1}\\
\Psi_{\xi}(z) ~:=~& \frac{\Gamma(2\beta/p)\Gamma(2\beta/q)}{\Gamma^2 (\beta)}\, \sum_{k=0}^{\infty}
\epsilon_k \bar{\xi}^k z^k,\\
\Upsilon_{\xi}(z) ~:=~& \sum_{k=0}^{\infty} a_k(\xi) \bar{\xi}^k z^k, \label{eqn:upsilon}
\end{align}
and
\begin{align*}
\epsilon_k :=& \frac{(2\beta/p)_k}{k!}\left\{\frac{\Gamma(k+2\beta)\Gamma(k+\beta)}
{\Gamma(k+ \beta + 2\beta/q) \Gamma(k+2\beta/p)}-1\right\}, \label{eqn:orderofck}\\
a_k(\xi):=& \frac {\left(\beta \right)_k} {k!} \left\{ \hyperg {2\beta/p - \beta}
{\beta +k } {2\beta+k} {|\xi|^2}  - \frac {\Gamma(2\beta/q) \Gamma(2\beta+k)} {\Gamma(\beta) \Gamma(\beta + 2\beta/q +k)}
\right\}.
\end{align*}

To see this, just use the formula
$$\left(a \right)_k=\frac{\Gamma(a+k)}{\Gamma(a)}$$
and the formula \eqref{eqn:binom}. Note that the functions above obviously depend also on $p$ and $\alpha$; we suppress the notation to keep the argument readable.

Since obviously
$$\|\Phi_{\xi}\|_{p,\alpha} = \dfrac{\Gamma(2\beta/p)\Gamma(2\beta/q)}
{\Gamma^2 (\beta)} \|f_{\xi}\|_{p,\alpha}$$
and by Corollary \ref{cor:rudin2},
$$\limsup_{|\xi|\to 1^-}\|f_{\xi}\|_{p,\alpha} =\infty,$$
we are done, once the following lemma is proved.

\begin{lemma}\label{lem:finalstep}
The functions $$\xi \mapsto \|\Psi_{\xi}\|_{p,\alpha}$$ and $$\xi \mapsto \|\Upsilon_{\xi}\|_{p,\alpha}$$ are bounded on $\disk$.
\end{lemma}

To prove this lemma, we shall use the following simple fact (See \cite[Lemma 2.2]{AVV} or \cite[Theorem 2]{AVV2}).

\begin{lemma}[L'H\^{o}pital Monotone Rule]\label{lem:monotonerule}
Let $-\infty<a<b<\infty$, and let $\varphi, \psi: [a, b]\to \mathbb{R}$ be continuous functions that are differentiable
on $(a,b)$, with $\varphi(a)=\psi(a)=0$ or  $\varphi(b)=\psi(b)=0$. Assume that $\psi^{\prime}(x)\neq 0$ for each $x$ in $(a, b)$.
If $\varphi^{\prime}/\psi^{\prime}$ is increasing (decreasing) on $(a, b)$, then so is  $\varphi/\psi$.
\end{lemma}

Also, We shall require a result of Hausdorff-Young type for $A_{\alpha}^p$, which is most likely known to the experts. However, we have been unable to find a reference, so we include a proof, for completeness.

\begin{lemma}[Hausdorff-Young theorem for $A_{\alpha}^p$]\label{lem:HLforBergman}
Suppose that $2\leq p<\infty$, $\alpha>-1$, and $\{a_k\}_{k=0}^{\infty}$ a sequence of complex numbers such that
\begin{equation*}
\sum_{k=0}^{\infty} \left\{\frac {k!\Gamma(1+\alpha)}{\Gamma(k+2+\alpha)}\right\}^{q-1} |a_k|^q <\infty.
\end{equation*}
Then the function $\varphi(z)=\sum_{k=0}^{\infty} a_k z^k$ is in $A_{\alpha}^p$, and
\begin{equation*}
\|\varphi\|_{p,\alpha}^q ~\leq \sum_{k=0}^{\infty} \left\{\frac {k!\Gamma(1+\alpha)}{\Gamma(k+2+\alpha)}\right\}^{q-1}
|a_k|^q.
\end{equation*}
\end{lemma}

\begin{proof}
For simplicity, we denote $\lambda_k:=\frac {k!\Gamma(1+\alpha)}{\Gamma(k+2+\alpha)}$.
Let $\mu$ be the discrete measure on the set $\mathbb{N}$ of nonnegative integers
which assigns the mass $\mu(k) = \lambda_k^{-1}$ to the integer $k = 0,1,2,\ldots$. Consider
the linear operator $T$ that maps the sequence $\mathbf{b}:=\{b_k\}_{k=0}^{\infty} = \{\lambda_k a_k\}_{k=0}^{\infty}$ to the
formal power series $\varphi(z) = \sum a_k z^k$. For $2 < p < \infty$, we want to show
that $T$ is bounded as an operator from $L^q(\mathbb{N},d\mu)$ to $L^p(\disk, dA_{\alpha})$, with norm
$\|T\|\leq 1$. But for $p = 2$ this follows from the relation
\[
\|\varphi\|_{2,\alpha}^2 = \sum_{k=0}^{\infty}\lambda_k |a_k|^2 = \sum_{k=0}^{\infty} \lambda_k^{-1} |b_k|^2=\|\mathbf{b}\|_{L^2(\mathbb{N},d\mu)}^2.
\]
For $p=\infty$, it is the trivial fact that
\[
\|\varphi\|_{\infty}\leq \sum_{k=0}^{\infty} |a_k| = \sum_{k=0}^{\infty} \lambda_k^{-1} |b_k|=\|\mathbf{b}\|_{L^1(\mathbb{N},d\mu)}.
\]
Thus we may invoke the Riesz-Thorin interpolation theorem (see \cite{Gra},
p.34, Theorem 1.3.4) to draw the conclusion that $\|T(\mathbf{b})\|_{p,\alpha} \leq \|\mathbf{b}\|_{L^q(\mathbb{N},d\mu)}$ for $2\leq p \leq \infty$.
\end{proof}

\begin{proof}[Proof of Lemma \ref{lem:finalstep}]

We start with the functions $\xi\mapsto \|\Psi_{\xi}\|_{p,\alpha}$.
By the asymptotic formula (\cite[p.47]{Erd})
\[
\frac {\Gamma(k+a)}{\Gamma(k+b)} ~\approx~ k^{a-b} \left\{ 1+\frac {1}{2k} (a-b)(a+b-1)+O(k^{-2})\right\},
\]
we see that
\begin{equation*}
\epsilon_k=O\left((k+1)^{2\beta-2\beta/q-2}\right).
\end{equation*}
In view of Lemma \ref{lem:HLforBergman}, this implies that
\begin{align*}
\|\Psi_{\xi}\|_{p,\alpha}^q ~\leq~& \left\{\frac{\Gamma(2\beta/p)\Gamma(2\beta/q)}{\Gamma^2 (\beta)}\right\}^q \sum_{k=0}^{\infty} \left\{\frac {k!\Gamma(2\beta-1)}{\Gamma(k+2\beta)}\right\}^{q-1}
\left|\epsilon_k |\xi|^k \right|^q \\
\lesssim~ & \sum_{k=0}^{\infty} (k+1)^{-(2\beta-1)(q-1)+2\beta q -2\beta -2q} |\xi|^{qk}\\
\lesssim~ & \sum_{k=0}^{\infty} (k+1)^{-q-1} < +\infty.
\end{align*}

As for the functions $\Upsilon_{\xi}$, we proceed as follows.
For $k\in \mathbb{Z}_{+}$, set
\[
g_k(x):=  \hyperg {\beta- 2\beta/q} {\beta +k } {2\beta+k} {x}.
\]
Note by \eqref{eqn:gauss} that
\[
g_k(1^-)=\frac {\Gamma(2\beta/q) \Gamma(2\beta+k)} {\Gamma(\beta) \Gamma(\beta + 2\beta/q +k)}.
\]
Thus we can rewrite \eqref{eqn:upsilon} as
\begin{equation}\label{eqn:newUps}
\Upsilon_{\xi}(z)=\sum_{k=0}^{\infty} \frac {\left(\beta \right)_k} {k!} \left[g_k(|\xi|^2)-g_k(1^-)\right] (z\bar{\xi})^k.
\end{equation}

Since we have assume $\alpha\neq (2-p)/(p-1)$, i.e., $\beta\neq q/2$,  the argument breaks down into two cases.

\subsubsection*{Case 1:  $1/2<\beta<q/2$}

We first show that
\begin{equation}\label{eqn:claim1}
|g_k(x)-g_k(1^{-})| \leq C_1(k,\beta) (1-x)^{2\beta/q}
\end{equation}
for all $x\in [0,1)$, where
\[
C_1(k,\beta):=\frac {(1-q/2) \Gamma(1-2\beta/q) } {\Gamma(1-2\beta/q+\beta)}
\frac {\Gamma(2\beta+k)} {\Gamma(\beta+k)}.
\]

By \eqref{eqn:diffhyperg}, we have
\begin{equation}\label{eqn:case1eq0}
g_k^{\prime}(x) = \frac {(\beta - 2\beta/q) \left(\beta +k\right) } {2\beta+k }
\hyperg {\beta - 2\beta/q + 1} {\beta +k+1 } {2\beta+k+1} {x}.
\end{equation}
Together with \eqref{eqn:euler}, this leads to
\begin{equation}\label{eqn:case1eq1}
\frac {g_k^{\prime}(x)} {(1-x)^{2\beta/q-1}}
~=~ \frac {(\beta - 2\beta/q) \left(\beta +k\right) } {2\beta+k }
\hyperg {\beta + 2\beta/q +k} {\beta} {2\beta+k+1} {x}.
\end{equation}
Note that the Taylor's coefficients of the last hypergeometric function are all positive. So the function
\[
x ~\mapsto~ \frac {g_k^{\prime}(x)} {(1-x)^{2\beta/q-1}}
\]
is increasing in the interval $[0,1)$, hence so is the function
\[
x ~\mapsto~ \frac {g_k(x)-g_k(1^{-})} {(1-x)^{2\beta/q}},
\]
by Lemma \ref{lem:monotonerule}. Therefore, for all $x\in [0,1)$,
\begin{align*}
\frac {g_k(x)-g_k(1^{-})} {(1-x)^{2\beta/q}} ~\leq~& \lim_{x\to 1^{-}} \frac {g_k(x)-g_k(1^{-})} {(1-x)^{2\beta/q}} \\
=~& \lim_{x\to 1^{-}} \frac {g_k^{\prime}(x)} {(-2\beta/q) (1-x)^{2\beta/q-1}}\\
=~& \frac {(1-q/2) \left(\beta +k\right) } {2\beta+k }
\frac {\Gamma(1+2\beta+k) \Gamma(1-2\beta/q)} {\Gamma(1+\beta+k) \Gamma(1-2\beta/q+\beta)},
\end{align*}
where the last equality follows from \eqref{eqn:case1eq1} and \eqref{eqn:gauss}. This yields \eqref{eqn:claim1}.

Now, we can apply Lemma \ref{lem:HLforBergman} to obtain
\begin{align*}
\|\Upsilon_{\xi}\|_{p,\alpha}^q ~\leq~& \sum_{k=0}^{\infty} \left\{\frac {k!\Gamma(2\beta-1)}{\Gamma(k+2\beta)}\right\}^{q-1} \left\{ \frac {\left(\beta \right)_k} {k!}
\left|g_k(|\xi|^2)-g_k(1^{-})\right| |\xi|^{k} \right\}^q\\
\leq~& (1-|\xi|^2)^{2\beta} \sum_{k=0}^{\infty} \left\{\frac {k!\Gamma(2\beta-1)}{\Gamma(k+2\beta)}\right\}^{q-1} \left\{ \frac {\left(\beta \right)_k} {k!}
C_1(k,\beta) \right\}^q |\xi|^{qk}\\
\lesssim~& (1-|\xi|^2)^{2\beta} \sum_{k=0}^{\infty} \frac {\Gamma(2\beta+k)}{\Gamma(2\beta) k!} |\xi|^{qk} ~\lesssim~ 1.
\end{align*}

\subsubsection*{Case 2: $\beta>q/2$}
Let $m:=\lceil  2\beta/q -  \beta \rceil$, where the `ceiling' function $\lceil t\rceil$ denotes the smallest integer that is greater than or equal to $t$.
Fix $k\in \mathbb{Z}_{+}$. By \eqref{eqn:diffhyperg},
we have
\begin{align*}
g_k^{(j)}(x) =& \frac {d^j}{dx^j} \hyperg {\beta - 2\beta/q}
{\beta +k } {2\beta+k} {x}\\
=&\frac {\left(\beta- 2\beta/q\right)_j \left(\beta +k\right)_j } {(2\beta+k)_j }
\hyperg {\beta - 2\beta/q+j} {\beta +k+j } {2\beta+k+j} {x}
\end{align*}
for $j=1,2,\ldots,m$.

The definition of $m$ implies that $m < 2\beta/q$, which guarantees that $g_k^{(j)}(1^{-})$, $j=1,\ldots,m$,  exist and are finite.
Moreover, by \eqref{eqn:gauss} and Stirling's formula, it is easy to check that
\begin{align}\label{eqn:orderofgj}
\left|g_k^{(j)}(1^{-})\right| ~=~& \left|(\beta-2\beta/q)_j\right| \frac { \Gamma(2\beta/q-j)} {\Gamma(\beta)}
\frac {\Gamma(\beta+j+k) \Gamma(2\beta+k)} {\Gamma(\beta+k)\Gamma(\beta+2\beta/q+k)} \\
\approx~& (k+1)^{\beta-2\beta/q+j} \notag
\end{align}
for $j=1,2,\ldots,m$. Also, note that the function
\[
x~\mapsto~ \hyperg {\beta - 2\beta/q +m} {\beta +k+m } {2\beta+k+m} {x}
\]
is increasing in the interval $[0, 1)$, since its Taylor coefficients are all positive.

It follows that
\begin{equation}\label{eqn:mthder}
\sup_{x\in [0,1)} \left|g_k^{(m)}(x)\right| = \left|g_k^{(m)}(1^{-})\right|.
\end{equation}
Hence, by Taylor's formula and \eqref{eqn:mthder}, \eqref{eqn:orderofgj}, we find that
\begin{align*}
|g_k(x)-g_k(1^{-})| ~\leq~& \sum_{j=1}^m  \frac {\left|g_k^{(j)}(1^{-})\right|} {j!} (1-x)^j\\
~\lesssim~& \sum_{j=1}^m  (k+1)^{\beta - 2\beta/q +j} (1-x)^j
\end{align*}
for all $x\in [0,1)$.
Again, we apply Lemma \ref{lem:HLforBergman} to obtain
\begin{align*}
\|\Upsilon_{\xi}\|_{p,\alpha}^q ~\leq~& \sum_{k=0}^{\infty}  \left\{\frac {k!\Gamma(2\beta-1)}{\Gamma(k+2\beta)}\right\}^{q-1}
\left\{ \frac {\left( \beta \right)_k} {k!} \left|g_k(|\xi|^2)-g_k(1^{-})\right| |\xi|^{qk} \right\}^q\\
\lesssim~& \sum_{j=1}^{m} (1-|\xi|^2)^{jq} \sum_{k=0}^{\infty} (k+1)^{jq-1} |\xi|^{2k}  ~\lesssim~ 1.
\end{align*}
This completes the proof.

\end{proof}

\section{Further remarks on the upper bound}

\noindent Recall that for the Riesz projection $P_+$ Hollenbeck and Verbitsky \cite{HV} managed to prove the optimal upper bound. Their argument is largely based on an elementary inequality
\begin{equation}\label{HolVer}
\max(|w|^p,|z|^p)\leq a_p |w+\overline{z}|^p-b_p\mathrm{Re}[(wz)^{p/2}],
\end{equation}
which holds for $p \in (1,2)$ for $a_p=\csc(\pi/p)^p$ and $b_p$ chosen appropriately.

Note that, if $f$ is a trigonometric polynomial, then setting $w=P_+ f$ and $z=\overline{(1-P_+)f}$ both $w$ and $z$ are analytic. In particular, then $\mathrm{Re}[(wz)^{p/2}]$ is subharmonic -- and this is the key point of the argument. Unfortunately, for the operators $P_\alpha$, the function $(1-P_\alpha)f$ is, in general, too irregular for an analogue of this argument to be applied in a straightforward manner. Proving (or disproving) the sharpness of our conjecture
\begin{equation}
 \|P_{\alpha}\|_{p,\alpha}=\frac{\Gamma(\frac {2+\alpha}{p})\Gamma(\frac {2+\alpha}{q})}{\Gamma^2 (\frac {2+\alpha}{2})},
\end{equation}
whether it is via an elementary inequality like \eqref{HolVer}, or by some other method, remains a problem for the future.

\subsubsection*{Acknowledgement}
This paper was written while the author was visiting the Department of
Mathematics and Statistics, University of Helsinki. He wishes to express his
gratitude for the warm hospitality he received there, especially from Professors 
Mats Gyllenberg, Tuomas Hyt¡§ onen, Pertti Mattila, Jari Taskinen and
Hans-Olav Tylli.



\end{document}